\documentclass[11pt]{amsart}
\usepackage{a4}
\usepackage{amsmath}
\usepackage{amsfonts}
\usepackage{amsthm}
\usepackage{amssymb}
\usepackage[usenames,dvipsnames,svgnames,table]{xcolor}
\usepackage{graphicx}
\usepackage{tikz}
\usepackage{tkz-graph}
\usepackage{tkz-berge}
\usetikzlibrary{arrows,shapes}
\usepackage{url}
\usepackage{pgf}
\usepackage{tikz}
\usepackage{algorithm}
\usepackage{algorithmic}

\usepackage{caption}
\usepackage{subcaption} 
\usepackage{float}  

\newtheorem{theorem}{Theorem}[section]

\newtheorem{prop}[theorem]{Proposition}
\newtheorem{claim}[theorem]{Claim}

\newtheorem{dfn}[theorem]{Definition}
\newtheorem{example}[theorem]{Example}

\newtheorem{rmk}[theorem]{Remark}

\newtheorem{note}[theorem]{Note}
\newtheorem{observations}[theorem]{Observation}

\def\thm{\textbftheorem}
\def \bp {\begin{prp} \ }
	\def \ep {\end{prp}}
\def \bpm {\begin{prm} \ }
	\def \epm {\end{prm}}
\def \bc {\begin{crl} \ }
	\def \ec {\end{crl}}
\def \thm {\begin{Theorem} \ }
	\def \ethm {\end{Theorem}}
\def \bl {\begin{lem} \ }
	\def \el {\end{lem}}
\def \bd {\begin{defi} \ \rm }
	\def \ed {\end{defi}}
\def \brm {\begin{rmk} \ }
	\def \erm {\end{rmk}}
\def \bxm {\begin{xmp} \ \rm }
	\def \exm {\end{xmp}}
\def \bcj {\begin{conj}}
	\def \ecj {\end{conj}}
\def \nmr {\begin{enumerate}}
	\def \enmr {\end{enumerate}}
\def \tmz {\begin{itemize}}
	\def \etmz {\end{itemize}}

\begin{document}
\title{Dom-forcing Sets in Graphs}
\maketitle
\begin{center}
 \author{\bf \sc Susanth P $^{1}$\  Charles Dominic $^{2} $ \ Premodkumar K P $^{3} $ \\{\footnotesize $^{1}$ Department of Mathematics},\\{\footnotesize Pookoya Thangal Memorial Government College,\\ Perinthalmanna, Kerala- 679322, INDIA.} \\{\footnotesize $^{2}$ Department of Mathematics},\\{\footnotesize CHRIST (Deemed to be University), \\Bengaluru-560029, INDIA.} \\{\footnotesize $^3$ Department of Mathematics},\\{\footnotesize Govt. College Malappuram ,\\ Kerala- 676509, INDIA.} \\ {\footnotesize $^{1}$ E-mail: $psusanth@gmail.com $}\\ { \hspace{2cm}\footnotesize $^{2}$ E-mail: $charles.dominic@christuniversity.in \newline \ \  charlesdominicpu@gmail.com$}\\
 {\footnotesize $^{3}$ E-mail: $pramod674@gmail.com  $}}   
\end{center}

\begin{abstract}
	A dominating set $D_{f}\subseteq V(G)$ of vertices in a graph $G$ is called a \emph{dom-forcing set} if the sub-graph induced by $\langle D_{f}  \rangle$  must form a zero forcing set. The minimum cardinality of such a set is known as the dom-forcing number of the graph $G$, denoted by $F_{d}(G)$.   This article embarks on an exploration of the dom-forcing number of a graph $G$. Additionally, it delves into the precise determination of $F_{d}(G)$  for certain well-known graphs.\\~\\   
 \textbf{AMS Subject Classification:} 05C50,05C69.\\
 \textbf{Key Words:} Zero forcing number, Domination number, Dom-forcing number. 
	\end{abstract}

\title{}

\section{Introduction}
Zero forcing is a step-by-step coloring process where at every discrete time step, a  black colored
vertex with a single white-colored neighbor forces that white-colored neighbor to become colored black. A zero forcing set of a simple graph $G$ is a set of initially colored black vertices that forces the entire graph $G$ to become colored black. The zero forcing number is the cardinality of the least zero forcing set. Zero Forcing on graphs was initiated in a workshop on linear algebra and graph theory organized by AIM in 2006 \cite{aim} and the concept was used to bound the minimum rank of a graph. The concept of zero-forcing was also used to study the quantum controllability of the system. Since its introduction zero forcing number has been a topic of interest in graph theory and a plethora of research has been carried out in this regard [5–9]. Zero forcing number of graph and its complement is studied in[10], it is used to study the logic circuit as well in [11].\\ 

A set $D\subseteq V $  of vertices in a graph $G=(V, E)$ is called a dominating set if every vertex $v\in V(G)$ is either an element of $D$ or is adjacent to an element of $D$.   A set $D_{f}\subseteq V$ of vertices is called a \emph{dom-forcing set} if it satisfies the following two conditions.\\
i)  $\langle D_{f}  \rangle $ must form a  dominating set.\\
ii) $\langle D_{f}  \rangle $ must form a  zero forcing set.\\

The minimum cardinality of such a set is called the dom-forcing number of the graph $G$ and is denoted by $F_{d}(G)$. For instance, contemplate the graph $C_5$ illustrated in Figure 1.  
\begin{center}
	\makeatletter
	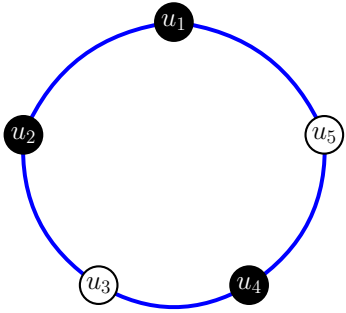
\begin{figure}[th]
		\centering
  \definecolor{ffffff}{rgb}{1,1,1}
		\begin{tikzpicture}[every node/.style={fill=red!60,circle,inner sep=1pt},
			.style={sibling distance=20mm,nodes={fill=red!45}},
			.style={sibling distance=20mm,nodes={fill=red!30}},
			.style={sibling distance=20mm,nodes={fill=red!25}}, style= thick]
			\draw [line width=1.5pt,color=blue,step=.5cm,] (3,2.5)to[bend right=30](1,1);
			\draw [line width=1.5pt,color=blue,step=.5cm,] (1,1)to[bend right=30](2,-1);
			
			\draw [line width=1.5pt,color=blue,step=.5cm,] (2,-1)to[bend right=30](4,-1);
			\draw [line width=1.5pt,color=blue,step=.5cm,] (4,-1)to[bend right=30](5,1);
			\draw [line width=1.5pt,color=blue,step=.5cm,] (5,1)to[bend right=30] (3,2.5);
			

			\node [draw,circle  ,fill=black,   text=ffffff, font=\huge, inner sep=0pt,minimum size=5mm] (3)  at (3,2.5)  {\scalebox{.5}{$u_{1}$}};	
			\node [draw,circle  ,fill=black,   text=ffffff, font=\huge, inner sep=0pt,minimum size=5mm] (3)  at (1,1)  {\scalebox{.5}{$u_{2}$}};
			\node [draw,circle  ,fill=ffffff,   text=black, font=\huge, inner sep=0pt,minimum size=5mm] (3)  at (2,-1)  {\scalebox{.5}{$u_{3}$}};
			\node [draw,circle  ,fill=black,   text=ffffff, font=\huge, inner sep=0pt,minimum size=5mm] (3)  at (4,-1)  {\scalebox{.5}{$u_{4}$}};
   \node [draw,circle  ,fill=ffffff,   text=black, font=\huge, inner sep=0pt,minimum size=5mm] (3)  at (5,1)  {\scalebox{.5}{$u_{5}$}};
		\end{tikzpicture}
  \caption {In this graph $C_5$, $D_f=\{u_1, u_2, u_4\}$ is a dominating as well as zero forcing set and no subset of the vertex set with cardinality less than three has this property so $F_d(C_5)=3$}
  \end{figure}
  \end{center}

In this article, we combine the concepts of domination and zero forcing and study a new graph-theoretic parameter called dom-forcing sets in graphs.\\  

\begin{note}
We use the  letter $Z$  to represent the zero forcing set. The zero forcing number of $G$ is denoted by $Z(G)$ and the domination number of $G$ is denoted by $\gamma (G)$.
\end{note}

The definition makes it evident that the combination of a zero forcing set and a dominating set constitutes a dom-forcing set. Therefore, the following relationship holds:
\begin{prop}
For any connected graph $G$\\ \begin{center} i) $ Z(G) \leq F_d(G)\leq Z(G)+\gamma (G)$\\
i) $ \gamma(G) \leq F_d(G)\leq Z(G)+\gamma (G)$
\end{center}
\end{prop}
The provided figure illustrates a graph where $Z(G)= \gamma (G)=  F_d( G)=2$.

\begin{example}
The dom-forcing set and the dom-forcing number of the graph in the Figure \ref{g} is  $ D_f=\{v_1, v_2\} $ and $ F_d( G)=2$.
\begin{figure}[h]

\definecolor{ffffff}{rgb}{1,1,1}
\begin{tikzpicture} [every node/.style={fill=red!60,circle,inner sep=1pt},
			.style={sibling distance=20mm,nodes={fill=red!45}},
			.style={sibling distance=20mm,nodes={fill=red!30}},
			.style={sibling distance=20mm,nodes={fill=red!25}}, style= thick]
\clip(-5.5,-0.5) rectangle (17.84,3.3);
\draw[line width=1.5pt,color=blue,step=.5cm,] (-0.32,2.66)-- (2.8,2.7);
\draw[line width=1.5pt,color=blue,step=.5cm,] (2.84,0.2)-- (2.8,2.7);
\draw[line width=1.5pt,color=blue,step=.5cm,] (2.84,0.2)-- (-0.28,0.14);
\draw[line width=1.5pt,color=blue,step=.5cm,] (-0.28,0.14)-- (-0.32,2.66);
\node [draw,circle  ,fill=black,   text=ffffff, font=\huge, inner sep=0pt,minimum size=5mm] (3)  at (-0.32,2.66)  {\scalebox{.5}{$v_{1}$}};
\node [draw,circle  ,fill=black,   text=ffffff, font=\huge, inner sep=0pt,minimum size=5mm] (4)  at (2.8,2.7)  {\scalebox{.5}{$v_{2}$}};
\node [draw,circle  ,fill=ffffff,   text=black, font=\huge, inner sep=0pt,minimum size=5mm] (11)  at (2.84,0.2)  {\scalebox{.5}{$v_{3}$}};
\node [draw,circle  ,fill=ffffff,   text=black, font=\huge, inner sep=0pt,minimum size=5mm] (3)  at (-0.28,0.14)  {\scalebox{.5}{$v_{4}$}};
\end{tikzpicture}
\caption {$G(V,E)$} 
\label{g}
\end {figure}
\end {example}
\begin{prop}
Let $G(V, E)$ be a graph and $S$ be its minimum zero forcing set. Then $F_d(G) \leq Z(G)+\gamma(G-N[S])$.
\end{prop}
\begin{proof}
    Let $G(V, E)$ be a graph and $S$ be its minimum zero forcing set. There exist a dominating set with cardinality $\gamma(G-N[S])$, say $D$, of $G-N[S]$. Then $S\cup D$ is a dom-forcing set of $G$. Therefore $F_d(G) \leq Z(G)+\gamma(G-N[S])$.
\end{proof}
It can be seen that this bound is sharp for Paths, Cycles etc. This is illustrated in the figure \ref{cyc1}. In next section, we discussed the exact values of the dome forcing number of some graphs with this bound is sharp and some with strict inequality.
\begin{figure}[h]
\definecolor{ffffff}{rgb}{1,1,1}
\begin{tikzpicture}

\draw[line width=1.5pt,color=blue,step=.5cm,] (5,6)-- (5,5);
\draw[line width=1.5pt,color=blue,step=.5cm,] (5,5)-- (5,4);
\draw[line width=1.5pt,color=blue,step=.5cm,] (5,4)-- (5,3);
\draw[line width=1.5pt,color=blue,step=.5cm,] (5,3)-- (5,2);
\draw[line width=1.5pt,color=blue,step=.5cm,] (5,2)-- (6,2);
\draw[line width=1.5pt,color=blue,step=.5cm,] (6, 2)-- (7,2);
\draw[line width=1.5pt,color=blue,step=.5cm,] (7,2)-- (8,2);
\draw[line width=1.5pt,color=blue,step=.5cm,] (8,2)-- (9,2);
\draw[line width=1.5pt,color=blue,step=.5cm,] (9,2)-- (9,3);
\draw[line width=1.5pt,color=blue,step=.5cm,] (9,3)-- (9,4);
\draw[line width=1.5pt,color=blue,step=.5cm,] (9,4)-- (9,5);
\draw[line width=1.5pt,color=blue,step=.5cm,] (9,5)-- (9,6);
\draw[line width=1.5pt,color=blue,step=.5cm,] (9,6)-- (8,6);
\draw[line width=1.5pt,color=blue,step=.5cm,] (8,6)-- (7,6);
\draw[line width=1.5pt,color=blue,step=.5cm,] (7,6)-- (6,6);
\draw[line width=1.5pt,color=blue,step=.5cm,] (6,6)-- (5,6);

\node [draw,circle  ,fill=black,   text=ffffff, font=\huge, inner sep=0pt,minimum size=5mm] (3)  at (5,6)  {\scalebox{.5}{$v_1$}};
\node [draw,circle  ,fill=black,   text=ffffff, font=\huge, inner sep=0pt,minimum size=5mm] (3)  at (5,5)  {\scalebox{.5}{$v_2$}};
\node [draw,circle  ,fill=ffffff,   text=black, font=\huge, inner sep=0pt,minimum size=5mm] (3)  at (5,4)  {\scalebox{.5}{$v_3$}};
\node [draw,circle  ,fill=ffffff,   text=black, font=\huge, inner sep=0pt,minimum size=5mm] (3)  at (5,3)  {\scalebox{.5}{$v_4$}};
\node [draw,circle  ,fill=black,   text=ffffff, font=\huge, inner sep=0pt,minimum size=5mm] (3)  at (5,2)  {\scalebox{.5}{$v_5$}};

\node [draw,circle  ,fill=ffffff,   text=black, font=\huge, inner sep=0pt,minimum size=5mm] (3)  at (6,2)  {\scalebox{.5}{$v_6$}};
\node [draw,circle  ,fill=ffffff,   text=black, font=\huge, inner sep=0pt,minimum size=5mm] (3)  at (7,2)  {\scalebox{.5}{$v_7$}};
\node [draw,circle  ,fill=black,   text=ffffff, font=\huge, inner sep=0pt,minimum size=5mm] (3)  at (8,2)  {\scalebox{.5}{$v_8$}};
\node [draw,circle  ,fill=ffffff,   text=black, font=\huge, inner sep=0pt,minimum size=5mm] (3)  at (9,2)  {\scalebox{.5}{$v_9$}};
\node [draw,circle  ,fill=ffffff,   text=black, font=\huge, inner sep=0pt,minimum size=5mm] (3)  at (9,3)  {\scalebox{.5}{$v_{10}$}};
\node [draw,circle  ,fill=black,   text=ffffff, font=\huge, inner sep=0pt,minimum size=5mm] (3)  at (9,4)  {\scalebox{.5}{$v_{11}$}};
\node [draw,circle  ,fill=ffffff,   text=black, font=\huge, inner sep=0pt,minimum size=5mm] (3)  at (9,5)  {\scalebox{.5}{$v_{12}$}};
\node [draw,circle  ,fill=ffffff,   text=black, font=\huge, inner sep=0pt,minimum size=5mm] (3)  at (9,6)  {\scalebox{.5}{$v_{13}$}};
\node [draw,circle  ,fill=black,   text=ffffff, font=\huge, inner sep=0pt,minimum size=5mm] (3)  at (8,6)  {\scalebox{.5}{$v_{14}$}};
\node [draw,circle  ,fill=ffffff,   text=black, font=\huge, inner sep=0pt,minimum size=5mm] (3)  at (7,6)  {\scalebox{.5}{$v_{15}$}};
\node [draw,circle  ,fill=ffffff,   text=black, font=\huge, inner sep=0pt,minimum size=5mm] (3)  at (6,6)  {\scalebox{.5}{$v_{16}$}};

\end{tikzpicture}
\caption{In the cycle $C_{16}$, $S=\{v_1,v_2\}$ be a minimum zero forcing set and $G-N[S]$ is a path with vertex set $\{v_4, v_5, \cdots , v_{15}\}$. $T=\{v_5, v_8, v_{11}, v_{14}\}$ be a minimum dominating set of $G-N[S]$, therefore $\gamma(G-N[S])=4$. Hence $S \cup T$ is a dom-forcing set of $C_{16}$, which is minimum, $F_d(C_{16})=6$ }
\label{cyc1}
\end{figure}
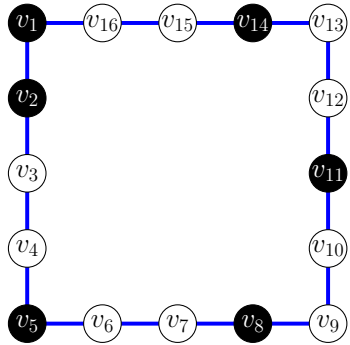

\section{Exact values of $F_d(G)$ }
Within this section, precise values of the dom-forcing number for several renowned graphs are presented. We start with a path $P_{n}$. 
\begin{theorem} \cite{pathd, zpr} \label{zpath}
For a path $P_n$, $Z(P_n)=1$ and $\gamma(P_n)=\lceil{n/3}\rceil$.
\end{theorem}
\begin{example}
Consider the graphs given in Figure \ref{g2}. For $P_3$, $D_f=\{v_1, v_3\}$ and $F_d=2$. For $P_4$, $D_f=\{u_1, u_4\}$ and $F_d=2$.We can generalize this result.
\begin{figure}[h]
\definecolor{ffffff}{rgb}{1,1,1}
\begin{tikzpicture}
\clip(-3.8,0) rectangle (12.5,3.5);
\draw [line width=1.5pt,color=blue,step=.5cm,](-1,2)-- (-3,2);
\draw [line width=1.5pt,color=blue,step=.5cm,](-1,2)-- (1,2);
\draw [line width=1.5pt,color=blue,step=.5cm,](3,2)-- (5,2);
\draw [line width=1.5pt,color=blue,step=.5cm,](5,2)-- (7,2);
\draw [line width=1.5pt,color=blue,step=.5cm,](7,2)-- (9,2);
\node [draw,circle  ,fill=black,   text=ffffff, font=\huge, inner sep=0pt,minimum size=5mm] (3)  at (-3,2)  {\scalebox{.5}{$v_{1}$}};
\node [draw,circle  ,fill=ffffff,   text=black, font=\huge, inner sep=0pt,minimum size=5mm] (3)  at (-1,2)  {\scalebox{.5}{$v_{2}$}};
\node [draw,circle  ,fill=black,   text=ffffff, font=\huge, inner sep=0pt,minimum size=5mm] (11)  at (1,2)  {\scalebox{.5}{$v_{3}$}};
\node [draw,circle  ,fill=black,   text=ffffff, font=\huge, inner sep=0pt,minimum size=5mm] (3)  at (3,2)  {\scalebox{.5}{$u_{1}$}};
\node [draw,circle  ,fill=ffffff,   text=black, font=\huge, inner sep=0pt,minimum size=5mm] (11)  at (5,2)  {\scalebox{.5}{$u_{2}$}};

\node [draw,circle  ,fill=ffffff,   text=black, font=\huge, inner sep=0pt,minimum size=5mm] (3)  at (7,2)  {\scalebox{.5}{$u_{3}$}};
\node [draw,circle  ,fill=black,   text=ffffff, font=\huge, inner sep=0pt,minimum size=5mm] (11)  at (9,2)  {\scalebox{.5}{$u_{4}$}};

\draw[color=black] (-0.68,1) node {$P_3$};
\draw[color=black] (6,1) node {$P_4$};
\end{tikzpicture}
\caption{dom-forcing set for $P_3$ and $P_4$ }
\label{g2}
\end{figure}
\end{example}
\begin{theorem}\label{dompath}
For a path $P_n$, $F_d(P_n)=\lfloor {n/3} \rfloor +1$. 
\end{theorem}
\begin{proof}
Consider a path $P_n$, with $V(P_n)=\{v_1, v_2, \cdots, v_n\}$. For $i= 2, 3, \cdots , n-1$ each vertex $v_i$ is adjacent with $v_{i-1}$ and $v_{i+1}$. It can be easily observe that $F_d(P_1)=F_d(P_2)=1$, for $n\geq 3$ consider the following cases and following subset of $V(P_n)$. \\ Case 1: Assume  $n \equiv 0,1 \mod 3$,  $D_f=\{v_{3k+1}, v_n\}$ where $0\leq k < \lfloor {n/3} \rfloor$.\\ Case 2: Assume $n\equiv 2 \mod 3$,  $D_f=\{v_{3k+1}, v_{n-1}\}$ where $0\leq k < \lfloor {n/3} \rfloor$. \\In both cases $D_f$ contains one of the end vertex, hence it forces $P_n$. Also, it dominates $P_n$, $D_f$ is a dom-forcing set. Any pendant vertex or two adjacent vertices forces $P_n$, no set with cardinality less than $|D_f|$ have this property. Therefore $F_d(P_n)=\lfloor {n/3} \rfloor +1$.
\end{proof}
We can find the following observation in \cite{zone}.
\begin{observations} \cite{zone} \label{one}
    For any connected graph $G = (V;E)$, $Z(G) = 1$ if and
only if $G = P_n$.
\end{observations}
 \begin{theorem}
     For any graph G, $F_d(G)=1$ if and only if $G= P_1$ or $  P_2$.
 \end{theorem}
 \begin{proof}
     For $G= P_1$ or $ P_2$, by theorem \ref{dompath}, $F_d(G)=1$. Conversely suppose that $F_d(G)=1$. Then $Z(G)=\gamma(G)=1$. But by observation \ref{one} $Z(G)=1$ implies that G is a path, also from theorem \ref{dompath}, $G= P_1$ or $  P_2$.
 \end{proof}
Now, let's examine the cycle graph $C_{n}$  and its dom-forcing number.  
\begin{theorem} \cite{pathd, zpr} \label{zycle}
For a Cycle $C_n$, of order $ n \geq 3 $, $Z(C_n)=2$ and $\gamma(C_n)=\lceil{n/3}\rceil$.
\end{theorem}

\begin{theorem} \label{dcycle}
For cycle $C_n$, $n\geq 3$ $$F_d(C_n)=\left \{ \begin {array}{ccl}  \lfloor {n/3} \rfloor +1 & if & n \equiv 0,1 \mod 3\\
\lfloor {n/3} \rfloor +2 & if & n \equiv 2 \mod 3\\
\end {array}\right .$$
\end{theorem} 

\begin{proof}
Let $C_n$ be a cycle with vertices $ v_1, v_2, \cdots, v_n$. Consider the following cases and following subset of $V(C_n)$. \\ Case 1: Assume $n \equiv 0,1 \mod 3$,  $D_f=\{v_{3k+1}, v_n\}$ where $0\leq k < \lfloor {n/3} \rfloor$.\\ Case 2: Assume  $n\equiv 2 \mod 3$,  $D_f=\{v_{3k+1}, v_{n-1},v_n\}$ where $0\leq k < \lfloor {n/3} \rfloor$.\\ In both cases $D_f$ contains two adjacent vertices, hence it forces $C_n$. Also for $i= 2, 3, \cdots, n-1$, each vertex $v_i$ is adjacent with $v_{i-1}$ and $v_{i+1}$, $D_f$ dominates $C_n$. Therefore  $D_f$ is a dom-forcing set. Any two adjacent vertices forces $C_n$, no set with cardinality less than $|D_f|$ have this property. Hence $$F_d(C_n)=\left \{ \begin {array}{ccl}  \lfloor {n/3} \rfloor +1 & if & n \equiv 0,1 \mod 3\\
\lfloor {n/3} \rfloor +2 & if & n \equiv 2 \mod 3\\
\end {array}\right .$$
\end{proof}
The ladder graph $L_n$ is the graph obtained by taking the
Cartesian product of $P_n$ with $P_2$.

\begin{theorem}\cite{ zpr}
 For a ladder graph $L_n$, $n\geq 2$, $Z(L_n)=2.$
\end{theorem}
\begin{theorem} \cite{doml}
 For a ladder graph $L_n$, $n\geq 2$, $\gamma(L_n)=\lfloor{\frac n2 +1 }\rfloor$.
\end{theorem}
\begin{theorem}
 For $n \geq 2 $, Let $L_n$ denote the Ladder graph, then $F_d(L_n)=\lceil{\frac n2}\rceil +1$.
\end{theorem}
\begin{proof}
Let $\{u_1, u_2, \cdots, u_n, v_1, v_2, \cdots, v_n\}$ be the vertices of the ladder graph $L_n$. For the ladder graphs $L_1, L_2, L_3$ and $L_4$, the subset of vertices $\{u_1\}, \{u_1,u_2\},\\ \{u_1,u_2, v_3\}$ and $\{u_1, u_2, v_4\}$ respectively constitute  a dom forcing  set. And for $n >4$, consider the following cases and the following subsets of vertex sets. \\Case 1: Assume that $\lceil \frac n 2 \rceil$ is odd, consider $$D_f=\{u_1, u_{4k-2}, v_{4k}, u_n\} \text{ for } 1\leq k <\lceil n/4 \rceil $$ Case 2:  Assume $ \lceil \frac n 2 \rceil$ is even, consider $$D_f=\{u_1,u_2, v_{4k}, u_{4k+2}, v_n\} \text{ for } 1\leq k <\lceil n/4 \rceil .$$ Now $D_f$ forms a
dom-forcing set. It can be easily verified that no set with less than $|D_f |$ vertices cannot form a dom-forcing set, hence $F_d(L_n)=\lceil \frac{n}{2} \rceil +1$.
\end{proof}

A Coconut tree graph $CT (m, n)$ is the graph obtained from the path $P_m$  by appending `$n$' new pendant edges at an end vertex of $P_m$.
\begin{theorem} [\cite{d2}]
For any coconut tree graph $CT(m,n)$, the domination number is $1+\lceil \frac{m-2}{3} \rceil$, where $m \geq 1, n\geq 1$.
\end{theorem}
\begin{theorem} 
For any coconut tree graph $CT(m,n)$, the zero forcing number is $n$, where $m \geq 1, n\geq 1$.
\end{theorem}
\begin{proof}
    Let $v_1, v_2, \cdots, v_m, u_1, u_2, \cdots, u_n$ be the vertices of $CT(m,n)$, then \\$\{ u_1, u_2, \cdots, u_n\}$ forms a zero forcing set. Which is minimum since $CT(m, n)$ contains a star graph of $n+2$ vertices.
\end{proof}
\begin{theorem} 
For any coconut tree $CT(m,n)$, the dom-forcing number is $n+ \lceil\frac{m-1}{3} \rceil$, where $m \geq 1, n\geq 1$.
\end{theorem}
\begin{proof}
    Let $v_1, v_2, \cdots v_m, u_1, u_2, \cdots u_n$ be the vertices of $CT(m,n)$. Then\\ $\{ u_1, u_2, \cdots, u_n\}$ forms a minimum zero forcing set. Case 1: Assume  $n =0,1 \mod 3$,  $D_f=\{ u_1, u_2, \cdots u_n\,v_{3k}\}$ where $0 < k \leq \lfloor {m/3} \rfloor$, and for $n=2 \mod 3$,  $D_f=\{u_1, u_2, \cdots u_n\,v_{3k},v_n\}$ where $0 < k \leq \lfloor {m/3} \rfloor$. In both cases $D_f$  is a dom-forcing set. No set with cardinality less than $|D_f|$ have this property. Therefore $\gamma (CT(m,n))=n+ \lceil\frac{m-1}{3} \rceil$.
\end{proof}
There are graphs having $Z(G)=\gamma (G)$, but dom-forcing number differ. For example, a diamond snake graph $D_n$ is a connected graph obtained from a path $P$ of length $n$ with each edge $e = (u, v)$ in P replaced by a cycle of length 4 with u and v as nonadjacent vertices of the cycle. For any diamond snake graph, $Z(D_n)=\gamma (D_n)=n+1$ \cite{z5,d2}.
\begin{theorem}
    For diamond snake graph $$F_d(D_n)=\left\{\begin{array}{cc}
        \frac{3n}{2} & \text{if n is even} \\
         \frac{3n+1}{2}& \text{if n is odd.}
    \end{array} \right.$$
\end{theorem}
\begin{proof}
    Let $u_1, \cdots,u_n, v_1, \cdots, v_{n+1}, w_1, \cdots, w_n$ be the vertices of diamond snake graph $D_n$(see Figure \ref{picd}). Consider the following cases and the following subset of the vertex set of $D_n$. \\Case 1: Assume  n is even, $D_f=\{u_1, \cdots,u_n,v_{2k}\}$ where $1 \leq k\leq \frac n 2$.\\ Case 2: Assume n is odd, $D_f=\{u_1, \cdots,u_n,v_{2k}\}$ where $1 \leq k\leq \frac {n+1}{2}$.\\ Now $D_f$ form a dom-forcing set. It can be easily verified that no set with less than $|D_f|$ vertices cannot form a dom-forcing set. Therefore  $$F_d(D_n)=\left\{\begin{array}{cc}
        \frac{3n}{2} & \text{if n is even} \\
         \frac{3n+1}{2}& \text{if n is odd.}
    \end{array} \right.$$
    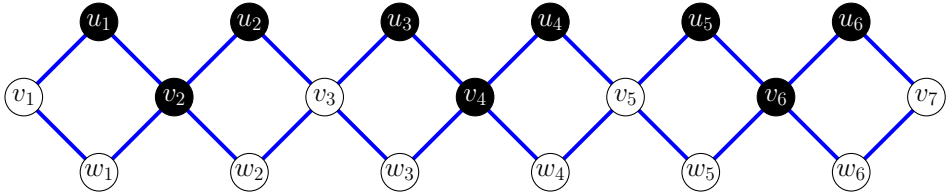
\begin{figure}[h] 
\definecolor{ffffff}{rgb}{1,1,1}
\begin{tikzpicture}
\draw[line width=1.5pt,color=blue,step=.5cm,] (0,0)-- (1,-1);
\draw[line width=1.5pt,color=blue,step=.5cm,] (2,0)-- (3,-1);
\draw[line width=1.5pt,color=blue,step=.5cm,] (4,0)-- (5,-1);
\draw[line width=1.5pt,color=blue,step=.5cm,] (6,0)-- (7,-1);
\draw[line width=1.5pt,color=blue,step=.5cm,] (8,0)-- (9,-1);
\draw[line width=1.5pt,color=blue,step=.5cm,] (10, 0)-- (11,-1);
\draw[line width=1.5pt,color=blue,step=.5cm,] (2,0)-- (1,-1);
\draw[line width=1.5pt,color=blue,step=.5cm,] (4,0)-- (3,-1);
\draw[line width=1.5pt,color=blue,step=.5cm,] (6,0)-- (5,-1);
\draw[line width=1.5pt,color=blue,step=.5cm,] (8,0)-- (7,-1);
\draw[line width=1.5pt,color=blue,step=.5cm,] (10,0)-- (9,-1);
\draw[line width=1.5pt,color=blue,step=.5cm,] (12, 0)-- (11,-1);
\draw[line width=1.5pt,color=blue,step=.5cm,] (0,0)-- (1,1);
\draw[line width=1.5pt,color=blue,step=.5cm,] (2,0)-- (3,1);
\draw[line width=1.5pt,color=blue,step=.5cm,] (4,0)-- (5,1);
\draw[line width=1.5pt,color=blue,step=.5cm,] (6,0)-- (7,1);
\draw[line width=1.5pt,color=blue,step=.5cm,] (8,0)-- (9,1);
\draw[line width=1.5pt,color=blue,step=.5cm,] (10, 0)-- (11,1);
\draw[line width=1.5pt,color=blue,step=.5cm,] (2,0)-- (1,1);
\draw[line width=1.5pt,color=blue,step=.5cm,] (4,0)-- (3,1);
\draw[line width=1.5pt,color=blue,step=.5cm,] (6,0)-- (5,1);
\draw[line width=1.5pt,color=blue,step=.5cm,] (8,0)-- (7,1);
\draw[line width=1.5pt,color=blue,step=.5cm,] (10,0)-- (9,1);
\draw[line width=1.5pt,color=blue,step=.5cm,] (12, 0)-- (11,1);

\node [draw,circle  ,fill=ffffff,   text=black, font=\huge, inner sep=0pt,minimum size=5mm] (3)  at (11,-1)  {\scalebox{.5}{$w_6$}};
\node [draw,circle  ,fill=ffffff,   text=black, font=\huge, inner sep=0pt,minimum size=5mm] (3)  at (9,-1)  {\scalebox{.5}{$w_5$}};
\node [draw,circle  ,fill=ffffff,   text=black, font=\huge, inner sep=0pt,minimum size=5mm] (3)  at (7,-1)  {\scalebox{.5}{$w_4$}};
\node [draw,circle  ,fill=ffffff,   text=black, font=\huge, inner sep=0pt,minimum size=5mm] (3)  at (5,-1)  {\scalebox{.5}{$w_3$}};
\node [draw,circle  ,fill=ffffff,   text=black, font=\huge, inner sep=0pt,minimum size=5mm] (3)  at (3,-1)  {\scalebox{.5}{$w_2$}};
\node [draw,circle  ,fill=ffffff,   text=black, font=\huge, inner sep=0pt,minimum size=5mm] (3)  at (1,-1)  {\scalebox{.5}{$w_1$}};
\node [draw,circle  ,fill=black,   text=ffffff, font=\huge, inner sep=0pt,minimum size=5mm] (3)  at (11,1)  {\scalebox{.5}{$u_{6}$}};
\node [draw,circle  ,fill=black,   text=ffffff, font=\huge, inner sep=0pt,minimum size=5mm] (3)  at (9,1)  {\scalebox{.5}{$u_{5}$}};
\node [draw,circle  ,fill=black,   text=ffffff, font=\huge, inner sep=0pt,minimum size=5mm] (3)  at (7,1)  {\scalebox{.5}{$u_{4}$}};
\node [draw,circle  ,fill=black,   text=ffffff, font=\huge, inner sep=0pt,minimum size=5mm] (3)  at (5,1)  {\scalebox{.5}{$u_{3}$}};
\node [draw,circle  ,fill=black,   text=ffffff, font=\huge, inner sep=0pt,minimum size=5mm] (3)  at (3,1)  {\scalebox{.5}{$u_{2}$}};
\node [draw,circle  ,fill=black,   text=ffffff, font=\huge, inner sep=0pt,minimum size=5mm] (3)  at (1,1)  {\scalebox{.5}{$u_1$}};
\node [draw,circle  ,fill=ffffff,   text=black, font=\huge, inner sep=0pt,minimum size=5mm] (3)  at (12,0)  {\scalebox{.5}{$v_7$}};
\node [draw,circle  ,fill=black,   text=ffffff, font=\huge, inner sep=0pt,minimum size=5mm] (3)  at (10,0)  {\scalebox{.5}{$v_6$}};
\node [draw,circle  ,fill=ffffff,   text=black, font=\huge, inner sep=0pt,minimum size=5mm] (3)  at (8,0)  {\scalebox{.5}{$v_5$}};
\node [draw,circle  ,fill=black,   text=ffffff, font=\huge, inner sep=0pt,minimum size=5mm] (3)  at (6,0)  {\scalebox{.5}{$v_4$}};
\node [draw,circle  ,fill=ffffff,   text=black, font=\huge, inner sep=0pt,minimum size=5mm] (3)  at (4,0)  {\scalebox{.5}{$v_3$}};
\node [draw,circle  ,fill=black,   text=ffffff, font=\huge, inner sep=0pt,minimum size=5mm] (3)  at (2,0)  {\scalebox{.5}{$v_2$}};
\node [draw,circle  ,fill=ffffff,   text=black, font=\huge, inner sep=0pt,minimum size=5mm] (3)  at (0,0)  {\scalebox{.5}{$v_1$}};
\end{tikzpicture}
\caption{The Diamond Snake Graph  $D_6$, dom-forcing set $D_f=\{ u_1, u_2, \cdots, u_6, v_2, v_4, v_6\}$ and $F_d(D_6)=9$}
\label{picd}
\end{figure}
\end{proof}

\section{Exact values of dom-forcing number of graphs where $Z(G)=F_d(G)$}

There are graphs with the dom-forcing number and the zero-forcing number that are the same. For example, consider the triangular snake graph, $TS_n$,  obtained from a path $P_{n+1}$ by replacing each edge of the path by a triangle $C_3$ \cite{z6}.
\begin{theorem}\cite{z6}
    For a triangular snake graph $TS_n$, $Z(TS_n)=n+1$.
\end{theorem}
\begin{theorem}
    For a triangular snake graph $TS_n$, $F_d(TS_n)=Z(TS_n)=n+1$.
\end{theorem}
\begin{proof}
    Let $v_1, v_2, \cdots, v_{n+1}, u_1, u_2,\cdots, u_{n}$ be the vertices of the triangular snake graph $TS_n$ (see Figure \ref{picts}). $D_f=\{v_1,u_1, u_2,\cdots, u_{n}\}$ form a dom-forcing set, which is minimum since $Z(TS_n)=n+1$. Therefore $F_d(TS_n)=Z(TS_n)=n+1$.
      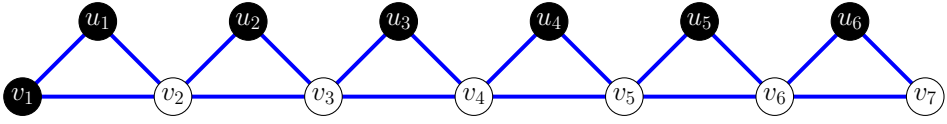
\begin{figure}[h] 
\definecolor{ffffff}{rgb}{1,1,1}
\begin{tikzpicture}
\draw[line width=1.5pt,color=blue,step=.5cm,] (0,0)-- (2,0);
\draw[line width=1.5pt,color=blue,step=.5cm,] (2,0)-- (4,0);
\draw[line width=1.5pt,color=blue,step=.5cm,] (4,0)-- (6,0);
\draw[line width=1.5pt,color=blue,step=.5cm,] (6,0)-- (8,0);
\draw[line width=1.5pt,color=blue,step=.5cm,] (8,0)-- (10,0);
\draw[line width=1.5pt,color=blue,step=.5cm,] (10, 0)-- (12,0);
\draw[line width=1.5pt,color=blue,step=.5cm,] (0,0)-- (1,1);
\draw[line width=1.5pt,color=blue,step=.5cm,] (2,0)-- (3,1);
\draw[line width=1.5pt,color=blue,step=.5cm,] (4,0)-- (5,1);
\draw[line width=1.5pt,color=blue,step=.5cm,] (6,0)-- (7,1);
\draw[line width=1.5pt,color=blue,step=.5cm,] (8,0)-- (9,1);
\draw[line width=1.5pt,color=blue,step=.5cm,] (10, 0)-- (11,1);
\draw[line width=1.5pt,color=blue,step=.5cm,] (2,0)-- (1,1);
\draw[line width=1.5pt,color=blue,step=.5cm,] (4,0)-- (3,1);
\draw[line width=1.5pt,color=blue,step=.5cm,] (6,0)-- (5,1);
\draw[line width=1.5pt,color=blue,step=.5cm,] (8,0)-- (7,1);
\draw[line width=1.5pt,color=blue,step=.5cm,] (10,0)-- (9,1);
\draw[line width=1.5pt,color=blue,step=.5cm,] (12, 0)-- (11,1);


\node [draw,circle  ,fill=black,   text=ffffff, font=\huge, inner sep=0pt,minimum size=5mm] (3)  at (11,1)  {\scalebox{.5}{$u_{6}$}};
\node [draw,circle  ,fill=black,   text=ffffff, font=\huge, inner sep=0pt,minimum size=5mm] (3)  at (9,1)  {\scalebox{.5}{$u_{5}$}};
\node [draw,circle  ,fill=black,   text=ffffff, font=\huge, inner sep=0pt,minimum size=5mm] (3)  at (7,1)  {\scalebox{.5}{$u_{4}$}};
\node [draw,circle  ,fill=black,   text=ffffff, font=\huge, inner sep=0pt,minimum size=5mm] (3)  at (5,1)  {\scalebox{.5}{$u_{3}$}};
\node [draw,circle  ,fill=black,   text=ffffff, font=\huge, inner sep=0pt,minimum size=5mm] (3)  at (3,1)  {\scalebox{.5}{$u_{2}$}};
\node [draw,circle  ,fill=black,   text=ffffff, font=\huge, inner sep=0pt,minimum size=5mm] (3)  at (1,1)  {\scalebox{.5}{$u_1$}};
\node [draw,circle  ,fill=ffffff,   text=black, font=\huge, inner sep=0pt,minimum size=5mm] (3)  at (12,0)  {\scalebox{.5}{$v_7$}};
\node [draw,circle  ,fill=ffffff,   text=black, font=\huge, inner sep=0pt,minimum size=5mm] (3)  at (10,0)  {\scalebox{.5}{$v_6$}};
\node [draw,circle  ,fill=ffffff,   text=black, font=\huge, inner sep=0pt,minimum size=5mm] (3)  at (8,0)  {\scalebox{.5}{$v_5$}};
\node [draw,circle  ,fill=ffffff,   text=black, font=\huge, inner sep=0pt,minimum size=5mm] (3)  at (6,0)  {\scalebox{.5}{$v_4$}};
\node [draw,circle  ,fill=ffffff,   text=black, font=\huge, inner sep=0pt,minimum size=5mm] (3)  at (4,0)  {\scalebox{.5}{$v_3$}};
\node [draw,circle  ,fill=ffffff,   text=black, font=\huge, inner sep=0pt,minimum size=5mm] (3)  at (2,0)  {\scalebox{.5}{$v_2$}};
\node [draw,circle  ,fill=black,   text=ffffff, font=\huge, inner sep=0pt,minimum size=5mm] (3)  at (0,0)  {\scalebox{.5}{$v_1$}};
\end{tikzpicture}
\caption{The triangular snake graph  $TS_6$, dom-forcing set $D_f=\{v_1, u_1, u_2, \cdots, u_6\}$ and $F_d(TS_6)=7$}
\label{picts}
\end{figure}
\end{proof}

A hypercube of dimension $k$, $Q_k$, has vertex set $\{0, 1\}^k$, with vertices adjacent when they differ in exactly one coordinate.
Or, the constructive definition:\\
Let $Q_0$ be a single vertex. For $k \geq 1$, $Q_k$ is formed by taking two copies of $Q_{k-1}$ and adding a matching joining the corresponding vertices in the two copies. This is equivalent to taking the Cartesian product of $Q_{k-1}$ and $K_2$ to form $Q_k$ \cite{hypd}. Zero forcing number of hypercube graph $Q_k$ is $2^{k-1}$ \cite{hypz}.
\begin{theorem}
    For a hypercube graph $Q_k$, $F_d(Q_k)=Z(Q_k)= 2^{k-1}$.
\end{theorem}
\begin{proof}
    Set of all vertices having first coordinate zero (all vertices of $Q_{k-1}$) form a zero forcing set with cardinality $2^{k-1}$. Clearly it dominates $Q_k$(see figure \ref{hypercube}). Hence $F_d(Q_k)=Z(Q_k)= 2^{k-1}$.
    \begin{figure}[h] 
\definecolor{ffffff}{rgb}{1,1,1}
\begin{tikzpicture}
\draw[line width=1.5pt,color=blue,step=.5cm,] (-2.04,3.96)-- (-0.04,3.96);
\draw[line width=1.5pt,color=blue,step=.5cm,] (-0.04,3.96)-- (-0.04,1.96);
\draw [line width=1.5pt,color=blue,step=.5cm,](-2.04,3.96)-- (-2.04,1.96);
\draw [line width=1.5pt,color=blue,step=.5cm,](-2.04,1.96)-- (-0.04,1.96);
\draw [line width=1.5pt,color=blue,step=.5cm,](-1.04,4.96)-- (-1.04,2.96);
\draw[line width=1.5pt,color=blue,step=.5cm,] (-1.04,4.96)-- (0.96,4.96);
\draw [line width=1.5pt,color=blue,step=.5cm,](0.96,4.96)-- (0.96,2.96);
\draw[line width=1.5pt,color=blue,step=.5cm,] (-1.04,2.96)-- (0.96,2.96);
\draw[line width=1.5pt,color=blue,step=.5cm,] (-0.04,3.96)-- (0.96,4.96);
\draw[line width=1.5pt,color=blue,step=.5cm,] (-1.04,4.96)-- (-2.04,3.96);
\draw [line width=1.5pt,color=blue,step=.5cm,](-1.04,2.96)-- (-2.04,1.96);
\draw [line width=1.5pt,color=blue,step=.5cm,](0.96,2.96)-- (-0.04,1.96);
\draw [line width=1.5pt,color=blue,step=.5cm,](4,5)-- (5.64,4.6);
\draw [line width=1.5pt,color=blue,step=.5cm,](5.64,4.6)-- (5.7,2.28);
\draw[line width=1.5pt,color=blue,step=.5cm,] (5.7,2.28)-- (3.96,2.8);
\draw [line width=1.5pt,color=blue,step=.5cm,](4,5)-- (3.96,2.8);
\draw[line width=1.5pt,color=blue,step=.5cm,] (4.68,3.6)-- (4.76,1.18);
\draw [line width=1.5pt,color=blue,step=.5cm,](3,4)-- (4.68,3.6);
\draw[line width=1.5pt,color=blue,step=.5cm,] (3,4)-- (3,2);
\draw[line width=1.5pt,color=blue,step=.5cm,] (3,2)-- (4.76,1.18);
\draw[line width=1.5pt,color=blue,step=.5cm,] (3.96,2.8)-- (3,2);
\draw [line width=1.5pt,color=blue,step=.5cm,](4,5)-- (3,4);
\draw [line width=1.5pt,color=blue,step=.5cm,](5.64,4.6)-- (4.68,3.6);
\draw[line width=1.5pt,color=blue,step=.5cm,] (5.7,2.28)-- (4.76,1.18);
\draw[line width=1.5pt,color=blue,step=.5cm,] (6.3,3.96)-- (8.14,3.58);
\draw [line width=1.5pt,color=blue,step=.5cm,](8.14,3.58)-- (8.12,1.3);
\draw[line width=1.5pt,color=blue,step=.5cm,] (8.12,1.3)-- (6.26,1.84);
\draw[line width=1.5pt,color=blue,step=.5cm,] (6.3,3.96)-- (6.26,1.84);
\draw[line width=1.5pt,color=blue,step=.5cm,] (7,5)-- (7.02,2.8);
\draw[line width=1.5pt,color=blue,step=.5cm,] (7.02,2.8)-- (8.76,2.4);
\draw [line width=1.5pt,color=blue,step=.5cm,](8.76,2.4)-- (8.82,4.7);
\draw[line width=1.5pt,color=blue,step=.5cm,] (8.82,4.7)-- (7,5);
\draw[line width=1.5pt,color=blue,step=.5cm,] (7,5)-- (6.3,3.96);
\draw[line width=1.5pt,color=blue,step=.5cm,] (8.82,4.7)-- (8.14,3.58);
\draw[line width=1.5pt,color=blue,step=.5cm,] (8.76,2.4)-- (8.12,1.3);
\draw[line width=1.5pt,color=blue,step=.5cm,] (7.02,2.8)-- (6.26,1.84);
\draw[line width=1.5pt,color=blue,step=.5cm,] (4,5)-- (7,5);
\draw[line width=1.5pt,color=blue,step=.5cm,] (5.64,4.6)-- (8.82,4.7);
\draw[line width=1.5pt,color=blue,step=.5cm,] (4.68,3.6)-- (8.14,3.58);
\draw[line width=1.5pt,color=blue,step=.5cm,] (3,4)-- (6.3,3.96);
\draw[line width=1.5pt,color=blue,step=.5cm,] (3.96,2.8)-- (7.02,2.8);
\draw[line width=1.5pt,color=blue,step=.5cm,] (5.7,2.28)-- (8.76,2.4);
\draw[line width=1.5pt,color=blue,step=.5cm,] (3,2)-- (6.26,1.84);
\draw[line width=1.5pt,color=blue,step=.5cm,] (4.76,1.18)-- (8.12,1.3);
\node [draw,circle  ,fill=black,   text=black, font=\huge, inner sep=0pt,minimum size=5mm] (3)  at (-2.04,3.96) {\scalebox{.5}{$ $}};
\node [draw,circle  ,fill=black,   text=black, font=\huge, inner sep=0pt,minimum size=5mm] (3)  at (-0.04,3.96) {\scalebox{.5}{$ $}};
\node [draw,circle  ,fill=ffffff,   text=black, font=\huge, inner sep=0pt,minimum size=5mm] (3)  at (-0.04,1.96) {\scalebox{.5}{$ $}};
\node [draw,circle  ,fill=ffffff,   text=black, font=\huge, inner sep=0pt,minimum size=5mm] (3)  at (-2.04,1.96) {\scalebox{.5}{$ $}};
\node [draw,circle  ,fill=black,   text=black, font=\huge, inner sep=0pt,minimum size=5mm] (3)  at (-1.04,4.96) {\scalebox{.5}{$ $}};
\node [draw,circle  ,fill=black,   text=black, font=\huge, inner sep=0pt,minimum size=5mm] (3)  at(0.96,4.96) {\scalebox{.5}{$ $}};
\node [draw,circle  ,fill=ffffff,   text=black, font=\huge, inner sep=0pt,minimum size=5mm] (3)  at (0.96,2.96) {\scalebox{.5}{$ $}};
\node [draw,circle  ,fill=ffffff,   text=black, font=\huge, inner sep=0pt,minimum size=5mm] (3)  at(-1.04,2.96) {\scalebox{.5}{$ $}};
\node [draw,circle  ,fill=black,   text=black, font=\huge, inner sep=0pt,minimum size=5mm] (3)  at (3,4) {\scalebox{.5}{$ $}};
\node [draw,circle  ,fill=black,   text=black, font=\huge, inner sep=0pt,minimum size=5mm] (3)  at (4.68,3.6) {\scalebox{.5}{$ $}};
\node [draw,circle  ,fill=ffffff,   text=black, font=\huge, inner sep=0pt,minimum size=5mm] (3)  at (4.76,1.18) {\scalebox{.5}{$ $}};
\node [draw,circle  ,fill=ffffff,   text=black, font=\huge, inner sep=0pt,minimum size=5mm] (3)  at (3,2) {\scalebox{.5}{$ $}};
\node [draw,circle  ,fill=black,   text=black, font=\huge, inner sep=0pt,minimum size=5mm] (3)  at (4,5) {\scalebox{.5}{$ $}};
\node [draw,circle  ,fill=black,   text=black, font=\huge, inner sep=0pt,minimum size=5mm] (3)  at (5.64,4.6) {\scalebox{.5}{$ $}};
\node [draw,circle  ,fill=ffffff,   text=black, font=\huge, inner sep=0pt,minimum size=5mm] (3)  at (5.7,2.28) {\scalebox{.5}{$ $}};
\node [draw,circle  ,fill=ffffff,   text=black, font=\huge, inner sep=0pt,minimum size=5mm] (3)  at (3.96,2.8) {\scalebox{.5}{$ $}};
\node [draw,circle  ,fill=black,   text=black, font=\huge, inner sep=0pt,minimum size=5mm] (3)  at (6.3,3.96) {\scalebox{.5}{$ $}};
\node [draw,circle  ,fill=black,   text=black, font=\huge, inner sep=0pt,minimum size=5mm] (3)  at (8.14,3.58) {\scalebox{.5}{$ $}};
\node [draw,circle  ,fill=ffffff,   text=black, font=\huge, inner sep=0pt,minimum size=5mm] (3)  at (8.12,1.3) {\scalebox{.5}{$ $}};
\node [draw,circle  ,fill=ffffff,   text=black, font=\huge, inner sep=0pt,minimum size=5mm] (3)  at (6.26,1.84) {\scalebox{.5}{$ $}};
\node [draw,circle  ,fill=black,   text=black, font=\huge, inner sep=0pt,minimum size=5mm] (3)  at (7,5) {\scalebox{.5}{$ $}};
\node [draw,circle  ,fill=ffffff,   text=black, font=\huge, inner sep=0pt,minimum size=5mm] (3)  at (7.02,2.8) {\scalebox{.5}{$ $}};
\node [draw,circle  ,fill=ffffff,   text=black, font=\huge, inner sep=0pt,minimum size=5mm] (3)  at (8.76,2.4) {\scalebox{.5}{$ $}};
\node [draw,circle  ,fill=black,   text=black, font=\huge, inner sep=0pt,minimum size=5mm] (3)  at (8.82,4.7) {\scalebox{.5}{$ $}};
\node [draw,circle  ,fill=black,   text=black, font=\huge, inner sep=0pt,minimum size=5mm] (3)  at (7,5) {\scalebox{.5}{$ $}};

\draw[color=black] (-1,1) node {$Q_3$};
\draw[color=black] (6,0) node {$Q_4$};

\end{tikzpicture}
\caption{dom-forcing (and zero forcing) in hypercube graphs $Q_3$  and $Q_4$}
\label{hypercube}
\end{figure}
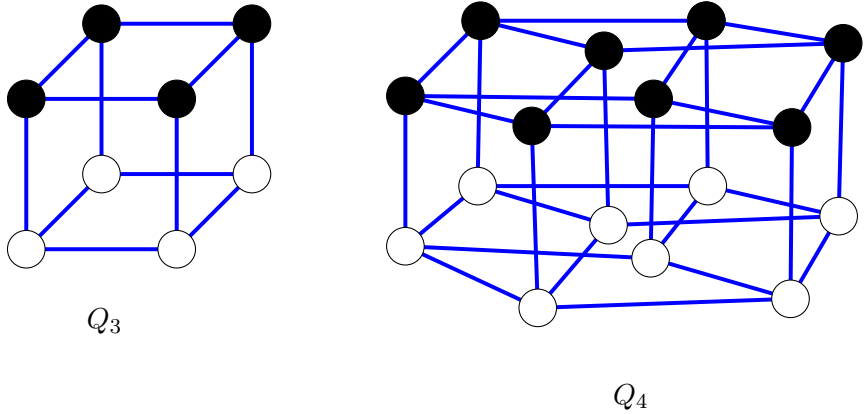
\end{proof}
\begin{theorem} \label{fdbd}
    Let G be a graph with $n$ vertices and $\Delta(G)=n-1$. Then $Z(G) \leq F_d(G)\leq Z(G)+1$.
\end{theorem}
\begin{proof}
    Let $v$ be the vertex of G having degree $n-1$. If the minimum zero forcing set $S$ contains $v$, then $Z(G)=F_d(G)$. Otherwise $S\cup \{v\}$ forms a dom-forcing set.
\end{proof}
 By using  theorem \ref{fdbd}, we can easily verify  the following results.
\begin{theorem}
    For a complete graph $K_n$, $Z(K_n)=F_d(K_n) =n-1$.
\end{theorem}
For a wheel graph, $W_n$, is a graph obtained by connecting a single vertex to all vertices of a cycle graph $C_{n-1}$.
\begin{theorem}
    For wheel graph $W_n$, $Z(W_n)=F_d(W_n)=3$.
\end{theorem}
For a dove tail graph, $DT_n$, is the graph $P_n + K_1, n \geq 2$. The dove tail graph has $n+1$ vertices and $2n-1$ edges \cite{dov}.
\begin{theorem}
    Let $DT_n$ denote the dove tail graph. Then $Z(DT_n)=F_d(DT_n)=2$.
\end{theorem}
In the case of complete bipartite graph every zero frcing set dominate the entire graph  and hence we have the following result.
\begin{theorem}
  Let $K_{m,n}$ be a complete bipartite where $n,m\geq 2$. Then $Z(K_{m,n})=F_d(K_{m,n}) =m+n-2$.
\end{theorem}
     From Theorems \ref{zpath}, \ref{dompath}, \ref{zycle} and \ref{dcycle} we have the following observations.

\begin{observations}
For a path $P_n$ and  a cycle $C_n$,
\begin{itemize}
    \item $Z(P_n)=F_d(P_n)$ if and only if $n=1$ or  $2$. 
    \item $Z(C_n)=F_d(C_n)$ if and only if $n=3$ or  $4$.
\end{itemize}
   
\end{observations}
\begin{theorem} \cite{ zpr}
For the star graph $K_{1,n}$, $\gamma (K_{1,n})=1$ and $Z(K_{1,n})=n-1$, for $n \geq 2$.
\end{theorem}
\begin{theorem}
For the star graph $K_{1,n}$, $F_d (K_{1,n})= n$, for $n \geq 2$.
\end{theorem}
\begin{proof}
We know that star graph $K_{1,n}$ have $n$ vertices of degree 1 and one vertex of degree $n$. Any minimum zero forcing set does not contain the vertex of degree $n$, and $Z(K_{1,n})=n-1$. Therefore by Theorem \ref{fdbd}, $F_d (K_{1,n})= n$. 

\end{proof}

The pineapple graph, denoted by $K_m^n$ is the graph formed
by coalescing any vertex of the complete graph $K_m$ with the star graph $K_{1,n}$, ($m\geq 3, n \geq 2$). The number of vertices in $K_m^n$ is $m+n$ and the number of edges in $K_m^n$ is $\frac{m^2-m+2n}{2}$ \cite{pinea}.

 \begin{theorem}
     Let $G$ be the pineapple graph $K_m^n$ with $n\geq 2, m \geq 3$. Then $F_d(G) =m+n-2$.
 \end{theorem}
 \begin{proof}
   We were aware that the pineapple graph comprises $m + n$ vertices. In \cite{z5}, the zero forcing number of the pineapple graph is $m+n-3$, when $n\geq 2, m \geq 3$. Pineapple graph has a vertex $v$ with degree $m+n-1$, and the vertex $v$ does not belong to any minimum zero forcing set. Therefore by Theorem \ref{fdbd}, $F_d(G) =m+n-3+1= m+n-2$.
 \end{proof}
 \begin{dfn} \cite{prope}
Let $G=(V, E)$ be a graph and $B$ a zero forcing set of $G$. Define $B^{(0)}=B$, and for $t \geq 0$, $B^{(t+1)}$ is the set of vertices $w$ for which there exists a vertex $b \in \bigcup _{s=0}^t B^{(s)}$ such that $w$ is the only neighbor of b not in $\bigcup _{s=0}^t B^{(s)}$. The propagation time of $B$ in $G$, denoted $pt(G, B)$, is the smallest integer $t_0$ such that $V =\bigcup _{s=0}^{t_0} B^{(s)}$.Two minimum zero forcing sets of the same graph may have different propagation times.The minimum propagation time of $G$ is
$$pt(G) = min\{pt(G, B)| B \text{ is a minimum zero forcing set of } G\}.$$
 \end{dfn}
 \begin{example}\label{pro}
 Let $G$ be the graph in figure \ref{propeg}. Let $B_1=\{v_1,v_2\}$ and $B_2=\{v_5,v_8\}$. Then $B_1^{(1)}=\{v_3, v_6\}$, $B_1^{(2)}=\{v_4\}$, $B_1^{(3)}=\{v_5\}$, $B_1^{(4)}=\{v_7\}$, and $B_1^{(5)}=\{v_8\}$, so $pt(G, B_1)=5$. However $B_2^{(1)}=\{v_7\}$,  $B_2^{(2)}=\{ v_6\}$, $B_2^{(3)}=\{v_1,v_4\}$, and $B_2^{(4)}=\{v_2, v_3\}$, so $pt(G, B_2)=4$.
     \begin{figure}[h]

\definecolor{ffffff}{rgb}{1,1,1}
\begin{tikzpicture} [every node/.style={fill=red!60,circle,inner sep=1pt},
			.style={sibling distance=20mm,nodes={fill=red!45}},
			.style={sibling distance=20mm,nodes={fill=red!30}},
			.style={sibling distance=20mm,nodes={fill=red!25}}, style= thick]

\draw[line width=1.5pt,color=blue,step=.5cm,] (0,4)-- (2,4);
\draw[line width=1.5pt,color=blue,step=.5cm,] (2,4)-- (3.04,2.46);
\draw[line width=1.5pt,color=blue,step=.5cm,] (3.04,2.46)-- (2,1);
\draw[line width=1.5pt,color=blue,step=.5cm,] (0,4)-- (-0.96,2.5);
\draw[line width=1.5pt,color=blue,step=.5cm,] (-0.96,2.5)-- (0,1);
\draw[line width=1.5pt,color=blue,step=.5cm,] (0,1)-- (2,1);
\draw[line width=1.5pt,color=blue,step=.5cm,] (3.04,2.46)-- (5.02,2.46);
\draw[line width=1.5pt,color=blue,step=.5cm,] (5.02,2.46)-- (7.02,2.46);

\node [draw,circle  ,fill=ffffff,   text=black, font=\huge, inner sep=0pt,minimum size=5mm] (3)  at (0,4)  {\scalebox{.5}{$v_4 $}};
\node [draw,circle  ,fill=ffffff,   text=black, font=\huge, inner sep=0pt,minimum size=5mm] (4)  at (2,4)  {\scalebox{.5}{$v_5 $}};
\node [draw,circle  ,fill=ffffff,   text=black, font=\huge, inner sep=0pt,minimum size=5mm] (11)  at (3.04,2.46)  {\scalebox{.5}{$v_6 $}};
\node [draw,circle  ,fill=ffffff,   text=black, font=\huge, inner sep=0pt,minimum size=5mm] (3)  at (2,1)  {\scalebox{.5}{$v_1 $}};

\node [draw,circle  ,fill=ffffff,   text=black, font=\huge, inner sep=0pt,minimum size=5mm] (3)  at (-0.96,2.5)  {\scalebox{.5}{$v_3 $}};
\node [draw,circle  ,fill=ffffff,   text=black, font=\huge, inner sep=0pt,minimum size=5mm] (4)  at (0,1)  {\scalebox{.5}{$v_2 $}};

\node [draw,circle  ,fill=ffffff,   text=black, font=\huge, inner sep=0pt,minimum size=5mm] (3)  at (5.02,2.46)  {\scalebox{.5}{$v_7 $}};

\node [draw,circle  ,fill=ffffff,   text=black, font=\huge, inner sep=0pt,minimum size=5mm] (11)  at (7.02,2.46) {\scalebox{.5}{$ v_8$}};

\end{tikzpicture}
\caption {Graph for Example \ref{pro} } 
\label{propeg}
\end {figure}
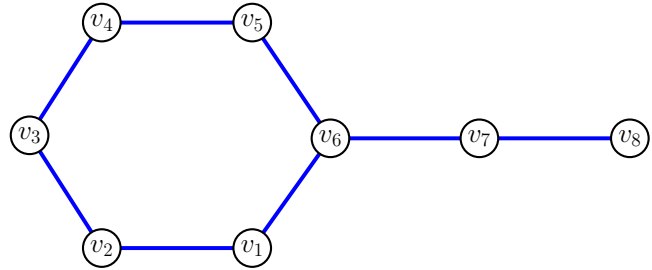
     
 \end{example}
 \begin{theorem} \label{prop1}
     Let $G=(V,E)$ be a graph with minimum propagation time one. Then $F_d(G)=Z(G)$.
 \end{theorem}
 \begin{proof}
     From the definition of minimum propagation time, there exist a minimum zero forcing set $B$ such that $V=B \cup B^{(1)}$. Every vertex of $B^{(1)}$ is adjacent to some vertex in $B$, so B dominates G. Hence $B$ is a dom-forcing set which is minimum. 
 \end{proof}
 The converse of the above theorem is false. For example let $G$ be the wheel graph with six vertices. Then $F_d(G)=Z(G)$ but minimum propagation time 2.
 \begin{theorem}
     Let P denote the Petersen graph shown in Figure \ref{peterson}. Then $F_d(P ) = 5 = Z(P )$.
 \end{theorem}
 \begin{proof}
     The minimum propagation time of the Petersen graph is one and minimum zero forcing number is five \cite{aim}. Hence by theorem \ref{prop1}, we have the result.
     
\begin{figure}[h]

\definecolor{ffffff}{rgb}{1,1,1}
\begin{tikzpicture} [every node/.style={fill=red!60,circle,inner sep=1pt},
			.style={sibling distance=20mm,nodes={fill=red!45}},
			.style={sibling distance=20mm,nodes={fill=red!30}},
			.style={sibling distance=20mm,nodes={fill=red!25}}, style= thick]

\draw[line width=1.5pt,color=blue,step=.5cm,] (-0.98,2.48)-- (1,4);
\draw[line width=1.5pt,color=blue,step=.5cm,] (3.02,2.52)-- (1,4);
\draw[line width=1.5pt,color=blue,step=.5cm,] (3.02,2.52)-- (2.2,0.34);
\draw[line width=1.5pt,color=blue,step=.5cm,] (2.2,0.34)-- (-0.18,0.34);
\draw[line width=1.5pt,color=blue,step=.5cm,] (-0.98,2.48)-- (-0.18,0.34);
\draw[line width=1.5pt,color=blue,step=.5cm,] (0.36,1.28)-- (-0.18,0.34);
\draw[line width=1.5pt,color=blue,step=.5cm,] (1.66,1.32)-- (2.2,0.34);
\draw[line width=1.5pt,color=blue,step=.5cm,] (1.96,2.32)-- (3.02,2.52);
\draw[line width=1.5pt,color=blue,step=.5cm,] (1,3)-- (1.66,1.32);
\draw[line width=1.5pt,color=blue,step=.5cm,] (1.66,1.32)-- (0.02,2.34);
\draw[line width=1.5pt,color=blue,step=.5cm,] (0.02,2.34)-- (1.96,2.32);
\draw[line width=1.5pt,color=blue,step=.5cm,] (1.96,2.32)-- (0.36,1.28);
\draw[line width=1.5pt,color=blue,step=.5cm,] (0.36,1.28)-- (1,3);
\draw[line width=1.5pt,color=blue,step=.5cm,] (1,3)-- (1,4);
\draw[line width=1.5pt,color=blue,step=.5cm,] (0.02,2.34)-- (-0.98,2.48);

\node [draw,circle  ,fill=black,   text=ffffff, font=\huge, inner sep=0pt,minimum size=5mm] (3)  at (1.96,2.32)  {\scalebox{.5}{$ $}};
\node [draw,circle  ,fill=ffffff,   text=black, font=\huge, inner sep=0pt,minimum size=5mm] (4)  at (1,4)  {\scalebox{.5}{$ $}};
\node [draw,circle  ,fill=black,   text=ffffff, font=\huge, inner sep=0pt,minimum size=5mm] (11)  at (0.36,1.28)  {\scalebox{.5}{$ $}};
\node [draw,circle  ,fill=ffffff,   text=black, font=\huge, inner sep=0pt,minimum size=5mm] (3)  at (2.2,0.34)  {\scalebox{.5}{$ $}};

\node [draw,circle  ,fill=black,   text=ffffff, font=\huge, inner sep=0pt,minimum size=5mm] (3)  at (1,3)  {\scalebox{.5}{$ $}};
\node [draw,circle  ,fill=ffffff,   text=black, font=\huge, inner sep=0pt,minimum size=5mm] (4)  at (-0.18,0.34)  {\scalebox{.5}{$ $}};
\node [draw,circle  ,fill=black,   text=ffffff, font=\huge, inner sep=0pt,minimum size=5mm] (11)  at (0.02,2.34) {\scalebox{.5}{$ $}};
\node [draw,circle  ,fill=black,   text=ffffff, font=\huge, inner sep=0pt,minimum size=5mm] (3)  at (1.66,1.32)  {\scalebox{.5}{$ $}};

\node [draw,circle  ,fill=ffffff,   text=black, font=\huge, inner sep=0pt,minimum size=5mm] (11)  at (-0.98,2.48) {\scalebox{.5}{$ $}};
\node [draw,circle  ,fill=ffffff,   text=black, font=\huge, inner sep=0pt,minimum size=5mm] (3)  at (3.02,2.52)  {\scalebox{.5}{$ $}};
\end{tikzpicture}
\caption {dom-forcing set and zero forcing set for the Petersen graph} 
\label{peterson}
\end {figure}
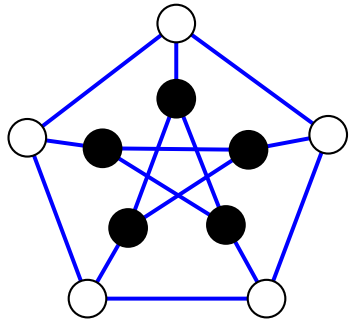
     
 \end{proof}
 
 \section{Graphs Where $\gamma(G)=F_d(G)$}
 There are some graphs with $\gamma(G)=F_d(G)$. For example a helm graph is a graph that is created by attaching a pendant edge to each vertex of an n-wheel graph's cycle, and is denoted by $H_m$, where $m \geq 4$ \cite{d4}.
\begin{theorem} \cite{d4}
    Let $H_m$ be a helm graph. Then the domination number of graph $H_m$ is $m$.
\end{theorem}
\begin{theorem}
    The dom-forcing number of helm graph $H_m$ is $m$. ie $F_d (H_m)=m=\gamma (H_m)$.
\end{theorem}
\begin{proof}
    Let $v_0$ denote the vertex in helm graph such that its degree is equal to $m$, and let $ v_1, v_2,\cdots, v_m$  represent the vertices in the helm graph, each with a degree of $3$. Also $u_1, u_2, \cdots, u_m$  represent the pendant vertices of the helm graph $H_m$ (See Figure \ref{pich}).  It can be easily verified  that $D_f=\{v_1, u_2, u_3, \cdots, u_m\}$ is a dom-forcing set of $H_m$, which is minimum since $\gamma (H_m)=m$. Therefore, $F_d (H_m)=m=\gamma (H_m)$.

    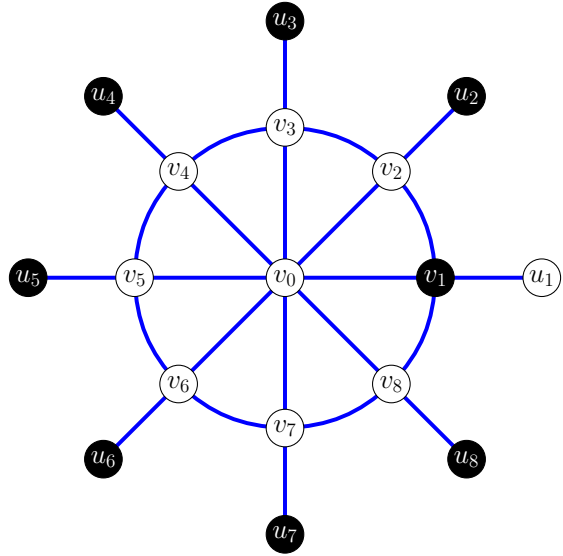
\begin{figure}[h] 
\definecolor{ffffff}{rgb}{1,1,1}
\begin{tikzpicture}
\draw[line width=1.5pt,color=blue,step=.5cm,] (0,0)-- (0,2);
\draw[line width=1.5pt,color=blue,step=.5cm,] (0,0)-- (2,0);
\draw[line width=1.5pt,color=blue,step=.5cm,] (0,0)-- (0,-2);
\draw[line width=1.5pt,color=blue,step=.5cm,] (0,0)-- (-2,0);
\draw[line width=1.5pt,color=blue,step=.5cm,] (0,0)-- (1.414,1.414);
\draw[line width=1.5pt,color=blue,step=.5cm,] (0, 0)-- (-1.414,1.414);
\draw[line width=1.5pt,color=blue,step=.5cm,] (0,0)-- (1.414,-1.414);
\draw[line width=1.5pt,color=blue,step=.5cm,] (0,0)-- (-1.414,-1.414);
\draw[line width=1.5pt,color=blue,step=.5cm,] (0,2)-- (0,3.414);
\draw[line width=1.5pt,color=blue,step=.5cm,] (1.414,1.414)-- (2.414,2.414);
\draw[line width=1.5pt,color=blue,step=.5cm,] (2,0)-- (3.414,0);
\draw[line width=1.5pt,color=blue,step=.5cm,] (0,-2)-- (0,-3.414);
\draw[line width=1.5pt,color=blue,step=.5cm,] (-2,0)-- (-3.414,0);
\draw[line width=1.5pt,color=blue,step=.5cm,] (-1.414,1.414)-- (-2.414,2.414);
\draw[line width=1.5pt,color=blue,step=.5cm,] (1.414,-1.414)-- (2.414,-2.414);
\draw[line width=1.5pt,color=blue,step=.5cm,] (-1.414,-1.414)-- (-2.414,-2.414);
\draw[line width=1.5pt,color=blue,step=.5cm,] (2,0)to[bend right=20] (1.414,1.414);
\draw[line width=1.5pt,color=blue,step=.5cm,] (1.414,1.414)to[bend right=20] (0,2);
\draw[line width=1.5pt,color=blue,step=.5cm,] (0,2)to[bend right=20] (-1.414,1.414);
\draw[line width=1.5pt,color=blue,step=.5cm,] (-2,0)to[bend right=20] (-1.414,-1.414);
\draw[line width=1.5pt,color=blue,step=.5cm,] (-1.414,1.414)to[bend right=20] (-2,0);
\draw[line width=1.5pt,color=blue,step=.5cm,] (0,-2)to[bend right=20] (1.414,-1.414);
\draw[line width=1.5pt,color=blue,step=.5cm,] (-1.414,-1.414)to[bend right=20] (0,-2);
\draw[line width=1.5pt,color=blue,step=.5cm,] (1.414,-1.414)to[bend right=20] (2,0);


\node [draw,circle  ,fill=ffffff,   text=black, font=\huge, inner sep=0pt,minimum size=5mm] (3)  at (0,0)  {\scalebox{.5}{$v_0$}};
\node [draw,circle  ,fill=black,   text=ffffff, font=\huge, inner sep=0pt,minimum size=5mm] (3)  at (2.414,-2.414)  {\scalebox{.5}{$u_{8}$}};
\node [draw,circle  ,fill=black,   text=ffffff, font=\huge, inner sep=0pt,minimum size=5mm] (3)  at (0,-3.414)  {\scalebox{.5}{$u_{7}$}};
\node [draw,circle  ,fill=black,   text=ffffff, font=\huge, inner sep=0pt,minimum size=5mm] (3)  at (-2.414,-2.414)  {\scalebox{.5}{$u_{6}$}};
\node [draw,circle  ,fill=black,   text=ffffff, font=\huge, inner sep=0pt,minimum size=5mm] (3)  at (-3.414,0)  {\scalebox{.5}{$u_{5}$}};
\node [draw,circle  ,fill=black,   text=ffffff, font=\huge, inner sep=0pt,minimum size=5mm] (3)  at (-2.414,2.414)  {\scalebox{.5}{$u_{4}$}};
\node [draw,circle  ,fill=black,   text=ffffff, font=\huge, inner sep=0pt,minimum size=5mm] (3)  at (0,3.414)  {\scalebox{.5}{$u_{3}$}};
\node [draw,circle  ,fill=black,   text=ffffff, font=\huge, inner sep=0pt,minimum size=5mm] (3)  at (2.4141,2.414)  {\scalebox{.5}{$u_{2}$}};
\node [draw,circle  ,fill=ffffff,   text=black, font=\huge, inner sep=0pt,minimum size=5mm] (3)  at (3.414,0)  {\scalebox{.5}{$u_1$}};
\node [draw,circle  ,fill=ffffff,   text=black, font=\huge, inner sep=0pt,minimum size=5mm] (3)  at (1.414,-1.414)  {\scalebox{.5}{$v_8$}};
\node [draw,circle  ,fill=ffffff,   text=black, font=\huge, inner sep=0pt,minimum size=5mm] (3)  at (0,-2)  {\scalebox{.5}{$v_7$}};
\node [draw,circle  ,fill=ffffff,   text=black, font=\huge, inner sep=0pt,minimum size=5mm] (3)  at (-1.414,-1.414)  {\scalebox{.5}{$v_6$}};
\node [draw,circle  ,fill=ffffff,   text=black, font=\huge, inner sep=0pt,minimum size=5mm] (3)  at (-2,0)  {\scalebox{.5}{$v_5$}};
\node [draw,circle  ,fill=ffffff,   text=black, font=\huge, inner sep=0pt,minimum size=5mm] (3)  at (-1.414,1.414)  {\scalebox{.5}{$v_4$}};
\node [draw,circle  ,fill=ffffff,   text=black, font=\huge, inner sep=0pt,minimum size=5mm] (3)  at (0,2)  {\scalebox{.5}{$v_3$}};
\node [draw,circle  ,fill=ffffff,   text=black, font=\huge, inner sep=0pt,minimum size=5mm] (3)  at (1.414,1.414)  {\scalebox{.5}{$v_2$}};
\node [draw,circle  ,fill=black,   text=ffffff, font=\huge, inner sep=0pt,minimum size=5mm] (3)  at (2,0)  {\scalebox{.5}{$v_1$}};

\end{tikzpicture}
\caption{The Helm graph $H_8$, dom-forcing set $D_f=\{v_1, u_2, u_3, \cdots, u_8\}$ and $F_d(H_8)=8$}
\label{pich}
\end{figure}
\end{proof}

 From theorem \ref{zpath}, \ref{dompath}, \ref{zycle} and \ref{dcycle} we can easily derive the following results
\begin{theorem}
    For any path $P_n$, $F_d(P_n)=\gamma(P_n)$ if and only if $n$ is not a multiple of three.
\end{theorem}
\begin{theorem}
    For any cycle $C_n$, $F_d(C_n)=\gamma(C_n)$ if and only if $n=3k+1$, where $k$ is any positive integer.
\end{theorem}

\section{dom-forcing number of splitting graph of a graph $G$ }
The splitting graph of a graph $G$ is the graph $S(G)$ obtained by taking a vertex $v'$ corresponding to each vertex $v \in G$ and join $v'$ to all vertices of $G$ adjacent to $v$ \cite{split}.
\begin{theorem}
    Let $G$ be a connected graph of order $n > 1$ and $S(G)$ be its splitting graph. Then $F_d[S(G)] \leq 2F_d(G)$.
\end{theorem}
\begin{proof}
    Consider any minimum dom-forcing set $D_f$ of $G$. Let $D_f=\{v_1, v_2, \cdots, v_k\}$ for $1\leq k<n$ be a minimum zero forcing set of $G$. Now consider the set $D'_f=\{v_1, v_2, \cdots, v_k\} \cup \{u_1, u_2, \cdots, u_k\}$, where $u_1, u_2, \cdots, u_k$ are vertices  corresponding to $v_1, v_2, \cdots, v_k$ which are added to obtain $S(G)$. As in \cite{zpr} $D'_f$ forces $S(G)$. From the definition of splitting graph $N(u_1, u_2, \cdots, u_k)= N(D_f)$, where $N(D_f)$ denotes the vertices which are adjacent to $\{v_1, v_2, \cdots, v_k\}$, and $N[v_1, v_2, \cdots, v_k]=V(S(G))-\{u_1, u_2, \cdots, u_k\}$. Hence $D'_f$ dominates $S(G)$, it is a dom-forcing set. Therefore, $F_d[S(G)] \leq 2F_d(G)$.
\end{proof}
\begin{dfn} \cite{z3}
    Given a graph $G =(V,E)$, a path cover is a set of disjoint induced paths in $G$ such that every vertex $v\in V$ belongs to exactly one path. The path cover number $P(G)$ is the minimum number of paths in a path cover.
\end{dfn}
\begin{theorem} \cite{z3}
    Let G be a graph. If $X \subset V (G)$ is a zero forcing set for G, then X induces a path cover for G and $P(G) \leq Z(G)$.
\end{theorem}
\begin{theorem} [\cite{pathsd, zpr}]
Let $S(P_n)$ be the splitting graph of the path $P_n$. Then $Z[S(P_n)]=2 Z(P_n)=2$ and $\gamma[S(P_n)]=\left \{ \begin {array}{ccl} \frac{n}{2} & if & n \equiv 0 \mod 4\\
\frac{n+1}{2} & if & n \equiv 1,3 \mod 4\\
\frac{n+2}{2} & if & n \equiv 2 \mod 4
\end {array}\right .$.
\end{theorem}
\begin{theorem}
    Let $S(P_n)$ be the splitting graph of the path $P_n$. Then for $2\leq n \leq 4$, $F_d[(S(P_n)]=n$.
\end{theorem}
\begin{proof}
    Let $v_1,v_2, \cdots, v_n $ be the vertices of the graph $P_n$ and $u_1, u_2, \cdots ,u_n$ be the vertices corresponding to $v_1,v_2, \cdots, v_n$ which are added to obtain $S(P_n)$. Depending upon the number of vertices of $P_n$ we consider the following subsets of the vertex set of $S(P_n)$.\\
    For $S(P_2)$, $D_f=\{u_1, v_1\}$, for $S(P_3)$, $D_f=\{u_1, v_1, v_2\}$, $D_f$ form a dom-forcing set which is minimum. For $S(P_4)$, $D_f=\{u_1, v_1, v_4, u_4\}$ form a dom-forcing set, therefore $F_d[S(P_4)]\leq 4$. The domination number of $S(P_4)$ is 2. Hence $2\leq F_d[S(P_4)]$. Let $A=\{v_2, v_3\}$ be a minimum dominating set of $S(P_4)$. Then $S(P_4)-A$ has four components (two components contains only single vertex and two of them is path of length one), hence the set A cannot induces a path cover for $S(P_4)$. Hence  $F_d[S(P_4)]$ cannot be 2. We claim the following
    \begin{claim}
        Any dominating set of cardinality 3 must contain A.
    \end{claim} 
   \noindent \textbf{Proof of the Claim:}
    If possible assume that there exists a dominating set $D$ of $S(P_4)$ which does not contain $A$ and the cardinality of $D$ is $3$. In $S(P_4)$, $\deg(u_1) = \deg(u_4)=1$, and $v_2, v_3$ be the vertices having degree 4. All other vertices are of degree 2. Then consider the following cases.\\
    \textbf{Case 1}: $D$ contains a vertex of degree 4.\\
    Then $D$ must contain any one of the pendent vertex and the remaining non-dominating vertices are of degree two and which are non-adjacent. Hence We cannot form a dominating set of cardinality $3$.\\
    \textbf{Case 2}: $D$ cannot contain a vertex of degree 4\\Then $D$ must contain the pendent vertices $u_1, u_4$, and all other vertices are of degree two. Hence we cannot form a dominating set of cardinality $3$. \\
    In both cases we get a contradiction, hence our assumption is wrong, and the claim is proved.\\~\\
    By the above claim $F_d[S(P_4)]$ cannot be 3. Therefore, we have $F_d[S(P_4)]=4$. 
    That is $$F_d[(S(P_n)]=n, \text {for } 2\leq n \leq 4.$$
\end{proof}

Let us now consider the case where the number of vertices in the splitting graph of the path is   
$\geq 5$.
\begin{theorem}
For $n\geq 5$, let $S(P_n)$ be the splitting graph of the path $P_n$. Then  $$ \begin {array}{ccl}\frac{n+2}{2} \leq F_d[(S(P_n)] \leq \frac{n+4}{2} & if & n \equiv 0 \mod 4\\
F_d[(S(P_n)]=\frac{n+3}{2} & if & n \equiv 1,3 \mod 4\\
F_d[(S(P_n)]=\frac{n+2}{2} & if & n \equiv 2 \mod 4\\
\end {array}$$
\end{theorem}
\begin{proof}
Let $v_1,v_2, \cdots, v_n $ be the vertices of the graph $P_n$ and $u_1, u_2, \cdots ,u_n$ be the vertices corresponding to $v_1,v_2, \cdots, v_n$ which are added to obtain $S(P_n)$. At least one vertex from each pair $u_1,v_2$ and $u_n, v_{n-1}$  belong to any dominating set, since $N(u_1)=\{v_2$\} and $N(u_n)=\{v_{n-1}\}$. Also at least one vertex from $v_{i-1}, v_{i+1}$ belong to the dominating set, since $N(u_i)=\{v_{i-1},v_{i+1}$\} and $N(v_i)=\{v_{i-1}, v_{i+1},u_{i-1}, u_{i+1}\}$. Thus any dominating set contains at least $\frac n 2$ number of vertices. The set $\{v_1, u_1\}$ forces $S(P_n)$. Now depending upon the number of vertices of $P_n$ we consider the following subsets of the vertex set of $S(P_n)$.\\
For $n \geq 5$ consider the following cases.\\\\
\textbf{Case 1}: Assume $n \equiv 0 \mod 4$,  $D_f=\{u_1, v_1, v_{4k},v_{4k+1} v_n, u_n\}$ where  $0 < k < \lceil{\frac{n}{4}}\rceil$ form a dom-forcing set, therefore $F_d[S(P_n)]\leq \frac{n+4}{2}$. The domination number of $S(P_n)$ is $\frac{n}{2}$. Hence $\frac{n}{2} \leq F_d[S(P_4)]$. Let A be a minimum dominating set of $S(P_n)$. Then $S(P_n)-A$ has $3+\frac{n}{4}$ components (two components contain only a single vertex, two of them is the path of length one and $\frac{n}{4}-1$ components having six vertices which include four pendent vertices), the set $A$ cannot induce a path cover for $S(P_n)$. Hence  $F_d[S(P_n)]$ cannot be $\frac{n}{2}$.  Therefore, we have $\frac{n+2}{2} \leq F_d[S(P_n)]\leq \frac{n+4}{2}$.\\\\
\textbf{Case 2}: Assume $n \equiv 1,2 \mod 4$,  $D_f=\{u_1, v_1, v_{4k},v_{4k+1}\}$ where  $0 < k < \lceil{\frac{n}{4}}\rceil$, form a dom-forcing set, therefore $$F_d[S(P_n)]\leq\left \{ \begin {array}{ccl} 
\frac{n+3}{2} & if & n \equiv 1 \mod 4\\
\frac{n+2}{2} & if & n \equiv 2 \mod 4\\
\end {array}\right . $$
If $n \equiv 2 \mod 4$, then the domination number of $S(P_n)$ is $\frac{n+2}{2}$. Hence $F_d[S(P_n)]=\frac{n+2}{2}=|D_f|$. If $n \equiv 1 \mod 4$, then the domination number of $S(P_n)$ is $\frac{n+1}{2}$. Hence $F_d[S(P_n)]\geq \frac{n+1}{2}$. Let A be a minimum dominating set of $S(P_n)$. Then $S(P_n)-A$ has $4+\frac{n-1}{4}$ components (three components contain only a single vertex, two of them are paths of length one and $\frac{n-1}{4}-1$ components having six vertices which include four pendent vertices), the set $A$ cannot induce a path cover for $S(P_n)$. Hence  $F_d[S(P_n)]$ cannot be $\frac{n+1}{2}$. Therefore, $F_d[S(P_n)]=\frac{n+3}{2}$.\\\\ \textbf{Case 3}: Assume $n \equiv 3 \mod 4$,  $D_f=\{u_1, v_1, v_{4k},v_{4k+1}, v_{n-1}\}$ where  $0 < k < \lceil{\frac{n}{4}}\rceil$ form a dom-forcing set, therefore $F_d[S(P_n)]\leq \frac{n+3}{2}$. The domination number of $S(P_n)$ is $\frac{n+1}{2}$. Hence $F_d[S(P_n)]\geq \frac{n+1}{2}$. Let A be a minimum dominating set of $S(P_n)$. Then $S(P_n)-A$ has at least $3+\frac{n-3}{4}$ components, the set A cannot induce a path cover for $S(P_n)$. Hence  $F_d[S(P_n)]$ cannot be $\frac{n+1}{2}$. Therefore, $F_d[S(P_n)]=\frac{n+3}{2}$.\\ 
From all the cases  we have, for $n\geq 5$, $$ \begin {array}{ccl}\frac{n+2}{2} \leq F_d[(S(P_n)] \leq \frac{n+4}{2} & if & n \equiv 0 \mod 4\\
F_d[(S(P_n)]=\frac{n+3}{2} & if & n \equiv 1,3 \mod 4\\
F_d[(S(P_n)]=\frac{n+2}{2} & if & n \equiv 2 \mod 4.\\
\end {array}$$
\end{proof}
\begin{theorem} \cite{pathsd, zpr}
Let $S(C_n)$ be the splitting graph of the cycle $C_n$. Then $Z(S(C_n))=2 Z(C_n)=4$ and $\gamma(S(C_n))=\left \{ \begin {array}{ccl} \frac{n}{2} & if & n \equiv 0 \mod 4\\
\frac{n+1}{2} & if & n \equiv 1,3 \mod 4\\
\frac{n+2}{2} & if & n \equiv 2 \mod 4
\end {array}\right .$.
\end{theorem}
\begin{theorem}
For $n \geq 4$, $F_d[S(C_n)] \leq \left \{ \begin {array}{ccl} \frac{n+4}{2} & if & n \equiv 0 \mod 4\\
\frac{n+5}{2} & if & n \equiv 1,3 \mod 4\\
\frac{n+6}{2} & if & n \equiv 2 \mod 4
\end {array}\right .$.
\end{theorem}
\begin{proof}
For $n\geq 4$, let $v_1,v_2, \cdots, v_n $ be the vertices of the graph $C_n$ and $u_1, u_2, \cdots ,u_n$ be the vertices corresponding to $v_1,v_2, \cdots, v_n$ which are added to obtain $S(C_n)$.We know that $\{v_1, v_2,u_2, u_3\}$ form a zero forcing set. Now depending upon the number of vertices of $C_n$ we consider the following cases and following subsets of the vertex set of $S(C_n)$.\\
Case 1: Assume $n \equiv 0,2,3 \mod 4$,  $D_f=\{u_2, u_3, v_{4k+1},v_{4k+2} \}$ where  $0 \leq k < \lceil{\frac{n}{4}}\rceil$  \\
Case 2: Assume $n \equiv 1 \mod 4$,  $D_f=\{u_2, u_3, v_{4k+1},v_{4k+2}, v_n\}$ where  $0 \leq k < {\frac{n-1}{4}}$.\\
In both cases $D_f$ form a dom-forcing set. Hence $$F_d[S(C_n)] \leq \left \{ \begin {array}{ccl} \frac{n+4}{2} & if & n \equiv 0 \mod 4\\
\frac{n+5}{2} & if & n \equiv 1,3 \mod 4\\
\frac{n+6}{2} & if & n \equiv 2 \mod 4.\\
\end {array}\right .$$
\end{proof}
\begin{theorem} \cite{zpr}
Let $S(K_{1,n})$ be the splitting graph of a Star Graph $K_{1,n}$. Then for $n \geq 2$,  $\gamma [S(K_{1,n})]=2$ and $Z[S(K_{1,n})]= 2n-2$.
\end{theorem}
\begin{theorem}
 Let $S(K_{1,n})$ be the splitting graph of a Star Graph $K_{1,n}$. Then for $n \geq 2$, $F_d [S(K_{1,n})]= 2n-1$.
\end{theorem}
\begin{proof}
For $n\geq 2$, let $u_1, v_1,v_2, \cdots, v_n $ be the vertices of the star graph $K_{1,n}$ with $\deg(u_1)=n$ and $u'_1, v'_1,v'_2, \cdots ,v'_n$ be the vertices corresponding to $u_1, v_1, v_2, \cdots, v_n$ which are added to obtain $S(K_{1,n})$. From the above theorem $Z[S(K_{1,n})]=2n-2$ and $B=\{v_2, \cdots, v_n,v'_2, \cdots ,v'_n \} $ form a zero forcing set which is minimum and any minimum zero forcing set does not contain $u_1$. $B$ not  a dom-forcing set. So adding  $u_1$ to $B$, we get a dom-forcing set which is minimum. Hence $F_d [S(K_{1,n})]= 2n-1$.
\end{proof}
\begin{theorem} \cite{zpr}
 Let $S(L_n)$ be the splitting graph of ladder graph $L_n$. Then for $n\geq 2$, $Z[S(L_n)]=4$. 
\end{theorem}
\begin{theorem}
 Let $S(L_n)$ be the splitting graph of ladder graph $L_n$. Then $\gamma[S(L_n)]= 2\lceil{\frac{n}{3}}\rceil$.
\end{theorem}
\begin{proof}
Let $\{u_1, u_2, \cdots, u_n, v_1, v_2, \cdots, v_n\}$ be the vertices of the ladder graph $L_n$ and let $\{u'_1, u'_2, \cdots, u'_n, v'_1, v'_2, \cdots, v'_n\}$ be the vertices corresponding to $\{u_1, u_2, \cdots, u_n, v_1, v_2, \cdots, v_n\}$ which are added to obtain $S(L_n)$. At least one vertex from $\{u'_1, u_2, v_1\}$,  $\{v'_1, v_2, u_1\}$, $\{u'_n, u_{n-1}, v_n\}$, and $\{v'_n, v_{n-1}, u_n\}$ belonging to any dominating set, since $N(u'_1)= \{ u_2, v_1\}$,  $N(v'_1)= \{v_2, u_1\}$, $N(u'_n)= \{ u_{n-1}, v_n\}$, and $N(v'_n)= \{ v_{n-1}, u_n\}$. And for $0< k< \lceil \frac {n}{3} \rceil$ at least one vertex from $\{u_{3k},u_{3k+2}\}$ and $\{v_{3k},v_{3k+2}\}$ belonging to any dominating set, since $$N(u_{3k+1})=\{u_{3k},u_{3k+2},v_{3k+1},u'_{3k},u'_{3k+2},v'_{3k+1}\},$$  $$N(v_{3k+1})=\{v_{3k},v_{3k+2},u_{3k+1},v'_{3k},v'_{3k+2},u'_{3k+1}\},$$ $$N(u'_{3k+1})=\{u_{3k},u_{3k+2},v_{3k+1}\},$$ $$N(v'_{3k+1})=\{v_{3k},v_{3k+2},u_{3k+1}\}.$$ Thus any dominating set contains at least $2\lceil \frac n 3 \rceil$ number of
vertices. For $S(L_1)$, $\{u_1, v_2\}$ be a minimum dominating set. For $n \geq 2$ depending upon the number of vertices of $L_n$ we consider the following subset.\\
For $n \equiv 0,2 \mod 3$,  $S=\{u_{3k+2},v_{3k+2} \}$ where  $0 \leq k < \lceil{\frac{n}{3}}\rceil$  and 
for $n \equiv 1 \mod 3$,  $S=\{ u_{3k+2},v_{3k+2}, u_n, v_n\}$ where  $0 \leq k < \lceil{\frac{n-1}{3}}\rceil$. Now $S$ form a dominating set.  So $\gamma[S(L_n)]= 2\lceil{\frac{n}{3}}\rceil$.
\end{proof}
\begin{theorem}
 For $n\geq 2$, $$2\lceil{\frac{n}{3}}\rceil \leq F_d[S(L_n)]\leq 2+2\lceil{\frac{n}{3}}\rceil.$$
\end{theorem}
\begin{proof}
Let $\{u_1, u_2, \cdots, u_n, v_1, v_2, \cdots, v_n\}$ be the vertices of the ladder graph $L_n$. And let $\{u'_1, u'_2, \cdots, u'_n, v'_1, v'_2, \cdots, v'_n\}$ be the vertices corresponding to $\{u_1, u_2, \cdots, u_n, v_1, v_2, \cdots, v_n\}$ which are added to obtain $S(L_n)$. For $S(L_1)$, $\{u_1, v_2\}$ be a minimum dom-forcing set. For $n \geq 2$ depending upon the number of vertices of $L_n$ we consider the following cases and following subsets of the vertex set of $S(L_n)$ .\\
Case 1: For $n \equiv 0,2 \mod 3$,  $D_f=\{u'_1, v'_1, u_{3k+2},v_{3k+2} \}$ where  $0 \leq k < \lceil{\frac{n}{3}}\rceil$  and \\
Case 2: For $n \equiv 1 \mod 3$,  $D_f=\{ u'_1, v'_1,u_{3k+2},v_{3k+2}, u_n, v_n\}$ where  $0 \leq k < \lceil{\frac{n-1}{3}}\rceil$.\\ Now $D_f$ form a dom-forcing set. Therefore $F_d[S(L_n)]\leq 2+2\lceil{\frac{n}{3}}\rceil$. Domination number of $S(L_n)$ is $2\lceil{\frac{n}{3}}\rceil$.    Hence $$2\lceil{\frac{n}{3}}\rceil \leq F_d[S(L_n)]\leq 2+2\lceil{\frac{n}{3}}\rceil.$$
\end{proof}

\section{Conclusion and open problems}

In the paper, we address the problem of determining the dom-forcing
number of graphs. In Section 2, we found precise values of the dom-forcing number for several renowned graphs. In Section 3, we determined the dom-forcing number of some graphs where $F_d(G) = Z(G)$. In Section 4, we have found the dom-forcing number of some graphs where $F_d(G)=\gamma(G)$. In section 5, we provide some inequalities containing the dom-forcing number of spitting graphs of some graphs.

There are few questions that remains open, for example see the following.
\begin{enumerate}
    \item Characterise the graph $G$ for which $F_d(G)=Z(G)$.
    \item Characterise the graph $G$ for which $F_d(G)=\gamma(G)$.
    \item Find the exact value of the dom-forcing number of splitting graphs of graphs like paths, cycles, ladder graphs etc.

    \item While we have determined 
 $F_d(G)$ for certain well-known graphs, a comprehensive characterization for a wider variety of graph classes remains an open challenge. Extending these results to other families of graphs, such as bipartite graphs, or random graphs, could yield valuable insights.
 \item Developing efficient algorithms to compute the dom-forcing number for arbitrary graphs is an important practical problem. Given the combinatorial complexity of both dominating sets and zero forcing sets, creating polynomial-time algorithms or approximation schemes would be highly beneficial.
 \item Exploring the relationships between  $F_d(G)$ and other graph parameters, such as,  chromatic number, independence number, and spectral properties, could uncover deeper connections and lead to a more unified theory of graph invariants. 
\end{enumerate}
\bibliographystyle{unsrt}  


\end{document}